\documentclass[12pt]{amsart}
\usepackage{amsmath,amsfonts,amsthm,amscd,textcomp,times, amssymb}
\usepackage[all,cmtip]{xy}
\usepackage[T1]{fontenc}
\usepackage{calligra}
\usepackage{multido}
\usepackage{frcursive}
\usepackage{mathrsfs}
\usepackage{tikz}
\usetikzlibrary{shapes,arrows}
\newtheorem{theorem}{Theorem}

\newtheorem{claim}{Claim}

\newtheorem{proposition}{Proposition}

\newtheorem{lemma}{Lemma}
\newtheorem{definition}{Definition}
\newtheorem{example}{Example}
\newtheorem{corollary}{Corollary}
\newtheorem{remark}{Remark}
\numberwithin{equation}{section}
\numberwithin{remark}{section}
\numberwithin{theorem}{section}
\numberwithin{proposition}{section}
\numberwithin{definition}{section}
\numberwithin{lemma}{section}
\numberwithin{claim}{section}
\numberwithin{corollary}{section}
\numberwithin{conjecture}{section}

\newcommand{\bull}{\ensuremath{{}\bullet{}}}
 
\newcommand{\cpn}{\ensuremath{\mathbb{P}^{N}}}
\newcommand{\slnc}{\ensuremath{SL(N+1,\mathbb{C})}}
\newcommand{\dlb}{\ensuremath{\overline{\partial}}}
\newcommand{\dl}{\ensuremath{\partial}}
\newcommand{\ra}{\ensuremath{\longrightarrow}}
\newcommand{\om}{\ensuremath{\omega}}
\newcommand{\vp}{\ensuremath{\varphi}}
\newcommand{\vps}{\ensuremath{\varphi_{\sigma}}}

\newcommand{\xhyp}{\ensuremath{X\times\mathbb{P}^{n-1}}}

\newcommand{\lambull}{\ensuremath{\lambda_{\bull}} }
\newcommand{\mubull}{\ensuremath{\mu_{\bull}} }
\newcommand{\elam}{\ensuremath{\mathbb{E}_{\lambda_{\bull}}}}
\newcommand{\emu}{\ensuremath{\mathbb{E}_{\mu_{\bull}}}}

\begin{document}
\bibliographystyle{alpha}
%%%%%%%%%%%%%%%%%%%%%%%%%%%%%%%%%%%%%%%%%%%%%%%%%%%%%%%%%%%%%%%%%%%%%%%
\title[Stable Pairs ]{Stable Pairs and Coercive Estimates for the Mabuchi Functional}
\author{Sean Timothy Paul}
\email{stpaul@math.wisc.edu}
\address{Mathematics Dept. Univ. of Wisconsin, Madison}
 %%%%%%%%%%%%%%%%%%%%%%%%%%%%%%%%%%%%%%%%%%%%%%%%%%%%%%%%%%%%%%%%%%%%%%%
\date{August 18, 2013}
 \vspace{-5mm}
\begin{abstract}{ We introduce the concept of a (semi)stable \emph{pair} and establish its numerical criterion. As an application we show that a linearly normal algebraic manifold $X\ra\cpn$ is stable if and only if the Mabuchi energy is proper on  the space $\mathcal{B}$ of Bergman metrics. As a further application we show that as soon as the Mabuchi energy of a polarized K\"ahler manifold $(X ,\mathbb{L})$ is proper on \emph{some} $\mathcal{B}$ the group $\mbox{Aut}(X ,\mathbb{L})$ is finite.}
 \end{abstract}
\maketitle
\setcounter{tocdepth}{1}
\tableofcontents  
 \section{Statement of Results}
Let $X^n\ra \cpn$ be a linearly normal projective variety. Recall that $R=R_X$ is the $X$-\emph{resultant} and $\Delta=\Delta_{\xhyp}$ is the $X$-\emph{hyperdiscriminant}\footnote{See Definitions \ref{cayleyform} , \ref{dual}, and equation \ref{hyper} . Undefined terms are explained below.}.
 
 \begin{definition}\label{cmss} 
\emph{Let $X\ra\cpn$ be an irreducible, $n$-dimensional, linearly normal complex projective variety. Then $X$ is \emph{\textbf {(semi)stable}} if and only if }
\begin{align}
(R^{\ \deg(\Delta)}\ , \ \Delta^{\deg(R)}) 
\end{align}
\emph{is a (semi)stable pair\footnote{See Definitions \ref{semistable} and \ref{stable}.} for the action of $G:=\slnc$.}
\end{definition}

Now we assume that $X\ra \cpn$ is a \emph{smooth} subvariety.  Let $\om_{FS}$ denote the Fubini Study metric on $\cpn$.  As usual $\nu_{\om}$ denotes the K-energy map of the restriction  $\om:={\om_{FS}}|_X$. Let $\mathcal{B}$ denote the Bergman metrics associated to this embedding. 

\begin{theorem}\label{mainresult}
\emph{ $X\ra \cpn$ is \emph{\textbf{stable}} if and only if the Mabuchi energy $\nu_{\om}$ is \emph{\textbf{proper}} on $\mathcal{B}$.}
\end{theorem}

  Recall that properness \emph{on} $\mathcal{B}$ means that there are constants $C_1>0$ and $C_2$ such that 
\begin{align}
\nu_{\om}(\vps)\geq C_1J_{\om}(\vps) +C_2\qquad \mbox{for all $\vps\in\mathcal{B}$} \ .
\end{align} 
For the definition of the $J_{\om}$ functional see Section \ref{append1}.  
\begin{remark}
\emph{The reader should compare Theorem \ref{mainresult} to Theorem D page 293 of \cite{paul2012}. The crucial improvement here is that we are able to get the inequality on \emph{all} of $\mathcal{B}$.}
\end{remark}

\begin{remark}   \emph{ In \cite{paul2012}  it is shown that the \emph{semi}stability of $X\ra\cpn$ is equivalent to a \emph{lower bound} for the Mabuchi energy restricted to $\mathcal{B}$.}
\end{remark}
 The following corollary should be compared with Theorem 5.4 from \cite{bbegz}. Those authors deal with singular Fano varieties whereas we treat an arbitrary polarized \emph{manifold}. 
  \begin{corollary}\label{finitegroup}
 \emph{Let $(X , \mathbb{L})$ be a polarized manifold. Let $h$ be a Hermitian metric on $\mathbb{L}$ and let $\om$ denote its curvature K\"ahler form. If the Mabuchi energy  $\nu_{\om}$ is proper on some $\mathcal{B}_k$ then $\mbox{Aut}(X , \mathbb{L})$ is a finite group.}
 \end{corollary}
This follows from Theorem \ref{mainresult} and Proposition \ref{finiteauto} .

Let $X$ be a Fano manifold. Let $\alpha(X)$ be Tian's \emph{alpha invariant}. Then we deduce the following from Corollary \ref{finitegroup}.
\begin{corollary}\label{alpha} 
\emph{ If $\alpha(X)>\frac{n}{n+1}$ then $\mbox{Aut}(X)$ is finite.}
\end{corollary}
\begin{remark}
\emph{The proof of Corollary \ref{alpha} does not use the existence of a K\"ahler Einstein metric, in particular we do not require the Matsushima theorem regarding reductivity of the Lie algebra of holomorphic vector fields. The reader should compare this to Corollary 1.5 from \cite{odaka&sano}.}
\end{remark}

\begin{corollary}
\emph{ Let $X\ra\cpn$ be an $n$-dimensional, linearly normal, complex projective manifold. Let $\om={\om_{FS}}|_X$. Then $\nu_{\om}$ is proper on $\mathcal{B}$ if and only if it is proper on all degenerations in $\mathcal{B}$.}
\end{corollary}

\begin{corollary}
\emph{ Let $X\ra\cpn$ be an $n$-dimensional, linearly normal, complex projective manifold. Let $\om={\om_{FS}}|_X$. Then $\nu_{\om}$ is bounded below on $\mathcal{B}$ if and only if it is bounded below on all degenerations in $\mathcal{B}$.}
\end{corollary}

All of these results follow from Theorem A of \cite{paul2012} and the \emph{numerical criterion} for pairs (see Theorems \ref{numericalcriterion} and \ref{summary}) established below.
 
 \begin{definition} \emph{A polarized manifold  $(X , \mathbb{L})$ is \emph{\textbf{asymptotically stable}} if and only if there is some $k_0\in \mathbb{N}$ such that for all $k\geq k_0$   
\begin{align}
X\xrightarrow{\mathbb{L}^k} \mathbb{P}^{N_k}
\end{align}
is stable with an exponent $m$ independent of $k$.}
\end{definition}
For the precise definition of stability see Section \ref{definitionofstability} Definition \ref{stable}.
\begin{corollary}
\emph{Let $X$ be a Fano manifold, polarized by (some power of ) $-K_X$. Assume that $\mbox{Aut}(X)$ is discrete and $X$ admits a K\"ahler Einstein metric. Then $X$ is asymptotically stable.}
\end{corollary}

 To state a converse to this we note that asymptotic stability is \emph{equivalent} to the existence of constants\footnote{As we will explain below the constant $A$ is basically the reciprocal of the ``stability exponent'' $m$.} $A>0$ and  $C_k$ such that for all $k\geq 1$ the inequality\footnote{For simplicity we just assume that $\mathbb{L}$ is very ample.}
\begin{align}
\nu_{\om}(\vps)\geq AJ_{\om}(\vps) +C_k\qquad \mbox{ $\vps\in\mathcal{B}_k$} \ 
\end{align} 
holds. We emphasize that $A$ is \emph{independent of $k$}. We summarize this discussion with the following corollary of Theorem \ref{mainresult}.

\begin{corollary}\label{existence}
\emph{If $(X , \mathbb{L})$ is asymptotically stable and the sequence $\{C_k\}_{k\geq 1}$ is bounded below, then the Mabuchi energy is proper on the entire space $\mathcal{H}_{\om}$ of K\"ahler metrics in the class $[\om]$. In particular, under this assumption, an asymptotically stable Fano manifold admits a K\"ahler Einstein metric.}
\end{corollary}
 \begin{remark}
 \emph{The existence asserted in Corollary \ref{existence}  relies on important results of Tian \cite{tian97} and Phong et al. \cite{pssw2008}.  See Section \ref{append1}.}
 \end{remark}

\begin{remark}\emph{The reader should consult the preprints \cite{cds1} and \cite{tian2012} for spectacular progress on the ``Tian-Yau-Donaldson'' conjecture in the Fano case.}
\end{remark}
 To close this section we should explain to the reader that all of our work on the stability conjectures was guided by the desire to take the results of \cite{dingtian}, \cite{tian94}, \cite{tian97} to their logical conclusion. Concerning properness of the Mabuchi functional, existence of K\"ahler Einstein metrics, and Stability,  Tian writes \footnote{This is a paraphrase.} (see \cite{tian97}, last paragraph on page 5) :
 \ \\
 \begin{center}
 \emph{Clearly, the Stability Conjecture holds provided one can deduce the properness  of $\nu_{\om}$ from the K-Stability of $(X , -K_X)$. However, this seems to be a highly nontrivial problem.}
  \end{center}

%%%%%%%%%%%%%%%%%%%%%%%%%%%%%%%%%%%%%%%%%%%%%%%%%%%%%%%%%%%%%%%%%%%%%%%%%%%%%%%%%%%%%%%%%
 \section{(Semi)stable Pairs}\label{definitionofstability}
In this section we give the precise definition of a (semi)stable pair. This definition will be put in context in Section \ref{equivexts} . To begin, we let $G$ denote any of the classical linear algebraic groups over $\mathbb{C}$.  Specifically $G$ can be taken to be any one of the following
\begin{align*}
SL(N+1) \ , \ SO(2N) \ ,\ SO(2N+1)\ , \ Sp(N) \ .
\end{align*}
Primarily we will be interested in the case when $G$ is the special linear group.
      
      For any vector space $\mathbb{V}$ and any $v\in \mathbb{V}\setminus\{0\}$ we  let $[v]\in\mathbb{P}(\mathbb{V})$ denote the line through $v$. If $\mathbb{V}$ and $\mathbb{W}$ are $G$ representations with (nonzero) points $v$ and $w$ respectively, we define the projective orbits :     
\begin{align}
\mathcal{O}_{vw}:=G\cdot [(v,w)]  \subset \mathbb{P}(\mathbb{V}\oplus\mathbb{W}) \ , \ \mathcal{O}_{v}:=G\cdot [(v,0)]  \subset \mathbb{P}(\mathbb{V}\oplus\{0\})\ .
\end{align}
We let $\overline{\mathcal{O}}_{vw}\ ,\  \overline{\mathcal{O}}_{v}$ denote the Zariski closures of these orbits. Throughout this paper we \emph{always} assume that 
\begin{align}
0\neq v\in\mathbb{V} \ , \ 0\neq w\in \mathbb{W} \ .
\end{align}
The central concept of this article is the following.
 \begin{definition}\label{semistable} 
\emph{The pair $(v,w)$ is \textbf{\emph{semistable}} if and only if}  $ \overline{\mathcal{O}}_{vw}\cap\overline{\mathcal{O}}_{v}=\emptyset $ .
  \end{definition} 
 
 \begin{remark}
\emph{The semistability of the pair $(v,w)$ depends only on $([v],[w])$. The reader should also observe that the definition is \emph{not} symmetric in $v$ and $w$. In virtually all examples where the pair $(v,w)$ is semistable $(w, v)$ is not semistable.}
\end{remark}
In order to define a \emph{strictly stable} (henceforth stable) pair we need a large (but fixed) integer $m$ and the auxiliary left regular representation of $G$
\begin{align}
G\times\mathcal{GL}(N+1,\mathbb{C}) \ \ni \ (\sigma , A)\ra \sigma\cdot A \ .
\end{align}
  Recall that $\mathcal{GL}(N+1,\mathbb{C})$ is the vector space of square matrices of size $N+1$. The action is matrix multiplication.

  Let $T$ denote any maximal algebraic torus of $G$.  Let $M_{\mathbb{Z}}=M_{\mathbb{Z}}(T)$ denote the {character lattice} of $T$
\begin{align}
M_{\mathbb{Z}}:= \mbox{Hom}_{\mathbb{Z}}(T,\mathbb{C}^*) \ . 
\end{align}

As usual, the dual lattice is denoted by $N_{\mathbb{Z}}$. It is well known that $ u\in N_{\mathbb{Z}}$ corresponds to an algebraic one parameter subgroup $\lambda^u$ of $T$. These are algebraic homomorphims $\lambda:\mathbb{C}^*\ra T$.  The correspondence is given by 
\begin{align}
(\cdot \ , \ \cdot) :N_{\mathbb{Z}}\times M_{\mathbb{Z}}\ra \mathbb{Z} \ , \ m(\lambda^u(\alpha))=\alpha^{(u , m)} \ .
\end{align}
 
 We introduce associated real vector spaces by extending scalars 
\begin{align} 
 \begin{split}
 &M_{\mathbb{R}}:= M_{\mathbb{Z}}\otimes_{\mathbb{Z}}\mathbb{R} \\
\ \\
& N_{\mathbb{R}}:= N_{\mathbb{Z}}\otimes_{\mathbb{Z}}\mathbb{R}\ .
\end{split}
\end{align}

Since $\mathbb{V}$ is rational it decomposes under the action of $T$ into  {weight spaces}
\begin{align}
\begin{split}
&\mathbb{V}=\bigoplus_{a\in {\mathscr{A}_T}}\mathbb{V}_{a}  \\
\ \\
& \mathbb{V}_{a}:=\{v\in \mathbb{V}\ |\ t\cdot v=a(t) v \ , \ t\in T\}
\end{split}
\end{align}
$\mathscr{A}_T$ denotes  the $T$-{support} of $\mathbb{V}$ 
\begin{align}
\mathscr{A}_T:= \{a \in M_{\mathbb{Z}}\ | \ \mathbb{V}_{a}\neq 0\} \ .
\end{align}
Given $v\in \mathbb{V}\setminus \{0\}$  the projection of $v$ into $\mathbb{V}_{a}$ is denoted by $v_{a}$. The support of any (nonzero) vector $v$ is then defined by
\begin{align}
\mathscr{A}_T(v):= \{a\in \mathscr{A}_T\ | \ v_{a}\neq 0\} \ .
\end{align}
\begin{definition} \emph{ Let $T$ be any maximal torus in $G$. Let $v\in \mathbb{V}\setminus\{0\}$ . The \textbf{\emph{weight polytope}} of $v$ is the compact convex lattice polytope $\mathcal{N}(v)$ given by}
\begin{align}\label{wtpolytope}
\mathcal{N}(v):=\mbox{ {\emph{conv}}}\ ( \mathscr{A}_T(v)) \subset M_{\mathbb{R}} \  , \ \mbox{ {\emph{conv}}}(\mathscr{A}):=\mbox{convex hull}(\mathscr{A}) \ .
\end{align}
\end{definition}
The \emph{\textbf{standard $N$-simplex}}, denoted by $Q_N$, is defined to be the weight polytope of  the identity operator 
\begin{align}
\mathbb{I}\in \mathcal{GL}(N+1,\mathbb{C})\ .
\end{align}
 
We remark that $Q_N$  is full-dimensional and contains the origin in its strict interior
\begin{align}
0\in Q_N:=\mathcal{N}(\mathbb{I})\subset M_{\mathbb{R}} \ .
\end{align}  

Let $\mathbb{V}$ be a $G$ module. We define the \emph{\textbf{degree}} of $\mathbb{V}$  as follows
\begin{align}
\deg(\mathbb{V}):=\min\Big\{k\in \mathbb{Z}_{>0} \ |\ \mathcal{N}(v)\subseteq kQ_N\ \mbox{for all $0\neq v\in \mathbb{V}$} \ \Big\} \ .
\end{align}

For any $v\in \mathbb{V}$ and $m\in \mathbb{N}$ we define
\begin{align}
v^m:=v^{\otimes m}\in \mathbb{V}^{\otimes m} \ .
\end{align}
Finally we can give the definition of a stable pair. The reader should compare this with Definition 2 on page 264 of \cite{paul2012}. 
\begin{definition}\label{stable} 
\emph{The pair $(v,w)$ is \emph{\textbf{stable}} if and only if there is a positive integer $m$ such that the pair
\begin{align}
(\mathbb{I}^q\otimes v^m \ , \ w^{m+1})
\end{align}
is  semistable where $q$ denotes the degree of $\mathbb{V}$. }
\end{definition}
\begin{remark}
\emph{If the pair $(v,w)$ is stable ``with exponent $m$'' , then it is also stable with \emph{any larger exponent}. See remark \ref{exponent} .} 
\end{remark}
We think of the pair
\begin{align}
(\mathbb{I}^q\otimes v^m \ , \ w^{m+1})
\end{align}
as being a ``perturbation'' of $(v,w)$. Formally letting $m\ra \infty$ we have that
\begin{align}
(\mathbb{I}^q\otimes v^m \ , \ w^{m+1})\ra (v\ ,\ w) \ .
\end{align}
In particular if $(v\ ,\ w)$ is stable then it is semistable, as we should expect.

In order to bring out the meaning of a (semi)stable pair we consider the case where $\mathbb{V}\cong \mathbb{C}$ (the trivial one dimensional representation) .  
\begin{proposition}\label{ss1}
\emph{\emph{The following statements are equivalent}.
\begin{align*}
& 1)\ (1,w) \ \mbox{is a semistable pair.} \\
\ \\
& 2)\ \mbox{  $0\notin \overline{G\cdot w}\subset \mathbb{W}$ .} \\
\end{align*}}
\end{proposition}
In order to state the next proposition we endow $\mathbb{W}$ with a Euclidean norm.
 \begin{proposition}\label{ss2}
\emph{\emph{The following statements are equivalent}.
\begin{align*}
& 1)\ (1,w) \ \mbox{is a stable pair.} \\
\ \\
& 2)\ \mbox{The orbit $G\cdot w\subset \mathbb{W}$ is closed and the stabilizer $G_w$ is finite.} \\
 \ \\
& 3)\ \mbox{There are constants $A>0$ and $B$ such that }\\
&  \log||\sigma\cdot w||^2\geq A\log||\sigma||_{hs}^2 + B \qquad \mbox{for all $\sigma\in G$} \ .
 \end{align*}}
\end{proposition}
  $||\sigma||_{hs}$ denotes the Hilbert-Schmidt norm of the matrix $\sigma$. 
\begin{remark}
\emph{ In both propositions we emphasize that the affine orbit is being considered.}
\end{remark}
The proof of Proposition \ref{ss1} can be safely left to the reader. Proposition \ref{ss2} is somewhat trickier and follows from Theorem \ref{summary}  below.
%%%%%%%%%%%%%%%%%%%%%%%%%%%%%%%%%%%%%%%%%%%%%%%%%%%%%%%%%%%%%%%%%%%%%%%%%%%%%%%%%%%%%%%%%%
  \section{Resultants, (Hyper)discriminants, and  the Stability of Projective Varieties }\label{resultants}
    In order to define our stability concept we first recall the definitions of  $\Delta(X)$, the hyperdiscriminant, and $R(X)$, the resultant, of a projective variety. We always assume that $X$ is embedded into $\cpn$ as a linearly normal variety. This insures that the resultant and discriminant of $X$ behave as well as possible (see \cite{tevelev} Section 1.4.3).  For further details and background we refer the reader to \cite{gkz} and \cite{paul2012} .
 
 Let $X^n\ra \cpn$ be an irreducible, $n$-dimensional, linearly normal, complex projective variety of degree $d\geq 2$.
Let $\mathbb{G}(k,N)$ denote the Grassmannian of $k$-dimensional \emph{projective} linear subspaces of $\cpn$. This is the same as $G(k+1, \mathbb{C}^{N+1})$ , the Grassmannian of $k+1$ dimensional subspaces of $\mathbb{C}^{N+1}$.   

\begin{definition}\label{cayleyform}
 (Cayley  \cite{cayley1860}) \emph{The \textbf{\emph{associated form}} of $X^n\ra \cpn$ is given by}
\begin{align}
Z_X:=\{L\in \mathbb{G}(N-n-1,N)\ | L\cap X\neq \emptyset \} \ .
\end{align}
\end{definition}
It is well known that $Z_X$ enjoys the following properties: \\
\ \\
$i)$   $Z_X$ is a \textbf{\emph{divisor}} in $\mathbb{G}(N-n-1,N)$ . \\
 \ \\
$ii)$   $Z_X$ is irreducible . \\
  \ \\
$iii)$  $\deg(Z_X)=d$ (in the Pl\"ucker coordinates) . \\
  \ \\
  
  Therefore there exists $R_X\in H^0(\mathbb{G}(N-n-1,N), \mathcal{O}(d))$ such that
 \begin{align}
 \{R_X=0\}=Z_X
 \end{align}
 $R_X$ is the Cayley-{Chow} form of $X$. Following  Gelfand, Kapranov, and Zelevinsky  we call $R_X$ the \textbf{\emph{X-resultant}} .  Observe that   
\begin{align}
R_X\in \mathbb{C}_{d(n+1)}[M_{(n+1)\times (N+1)}]^{SL(n+1,\mathbb{C})} \ . 
\end{align}

\subsection{Discriminants} We assume that $X\ra\cpn$ has degree $d\geq 2$. Let $X^{sm}$ denote the smooth points of $X$. For $p\in X^{sm}$ let 
$\mathbb{T}_p(X)$ denote the {embedded} tangent space to $X$ at $p$. Recall that $\mathbb{T}_p(X)$ is an $n$-dimensional {projective} linear subspace of $\cpn$.
\begin{definition}\label{dual}
 \emph{ The \emph{\textbf{dual variety}} of $X$, denoted by $X^{\vee}$, is the Zariski closure of the set of  {tangent hyperplanes} to $X$ at its smooth points }
\begin{align}
X^{\vee}:=\overline{ \{ f\in {\cpn}^{\vee} \ |  \  \mathbb{T}_p(X)\subset \ker(f) \ , \ p\in X^{sm}\} } \ .
\end{align}
\end{definition}
 Usually $X^{\vee}$ has codimension one in $ {\cpn}^{\vee}$.  
\begin{definition} \emph{The \emph{\textbf{dual defect}} of $X\ra \cpn$ is the integer}
\begin{align}
\delta(X):=\mbox{\emph{codim}}(X^{\vee})-1 \ .
\end{align}
\end{definition}
 When $\delta=0$ there exists an irreducible homogeneous polynomial $\Delta_X\in \mathbb{C}[{\cpn}^{\vee}] $ (  the  \textbf{\emph{X-discriminant}})  such that
\begin{align}
X^{\vee}=\{\Delta_X=0\} \ .
\end{align}
    
 %%%%%%%%%%%%%%%%%%%%%%%%%%%%%%%%%%%%%%%%%%%%%%%%%%%%%%%%%%%%%%%%%%%%

\subsection{Hyperdiscriminants }  Given $X\ra \cpn$ we consider the \emph{Segre embedding} 
\begin{align}
\xhyp\ra \mathbb{P}({M_{n\times (N+1)}}^{\vee}) \ .
\end{align}
Of basic importance for this paper is the next proposition which follows from work of Weyman and Zelevinsky (see \cite{weymanzelevinsky}) and Zak ( see \cite{zak} ) .
\begin{proposition}\label{cayleytrick} \emph{ Let $X\ra \cpn$ be a nonlinear subvariety embedded by a very ample complete linear system. Then $\delta(\xhyp)=0 $ .  }
 \end{proposition}
\begin{remark} \emph{The reader should observe that $X$ is not required to be smooth .}
\end{remark}

It follows from Proposition \ref{cayleytrick} that there exists a nonconstant homogeneous polynomial 
 \begin{align}\label{hyper}
 { \Delta_{\xhyp}\in \mathbb{C}[M_{n\times (N+1)}]}^{SL(n,\mathbb{C})}\ ,
 \end{align}

which we shall call the \textbf{\emph{X-hyperdiscriminant}}, such that
\begin{align}
\{\Delta_{\xhyp}=0\}=(\xhyp)^{\vee} \ .
\end{align}
To save space we shall sometimes let
\begin{align}
R(X):=R_X \ , \ \Delta(X):= \Delta_{\xhyp} \ .
\end{align}

\begin{remark}
\emph{For further information on the hyperdiscriminant the reader is referred to \cite{paul2012} Section 2.2 pg. 270. The two crucial properties are that $\xhyp$ is always dually non-degenerate and that $\Delta_{\xhyp}$ encodes only the Ricci curvature of $X\ra \cpn$.}
\end{remark}

The obvious task is to compute the degree of this polynomial .

\begin{proposition}\label{degreeofhyp} (see \cite{paul2012} Proposition 5.7) \emph{Assume $X$ is {smooth}. Then the degree of the hyperdiscriminant is given as follows}
\begin{align}
\deg(\Delta):=\deg(\Delta_{\xhyp}) =n(n+1)d-d\mu \in \mathbb{Z}_+\ .
\end{align}
\end{proposition} 
In the preceding proposition $\mu$ denotes, as usual, the average of the scalar curvature of ${\om_{FS}}|_X$.  For the algebraic geometer this number is essentially the subdominant coefficient of the Hilbert polynomial of $X$. 

So far, to a nonlinear projective variety $X\ra\cpn$ we have associated two polynomials: $R(X)$ and $\Delta(X)$.
 Translation invariance of the Mabuchi energy forces us to {normalize the degrees} of  these polynomials. From this point on we are interested in the pair 
\begin{align} 
\big(  R(X)^{\deg(\Delta )}  \ ,\  \Delta (X)^{\deg(R )}  \big)\ .
  \end{align}
 Below we shall let $r$ denote their \emph{common} degree 
 \begin{align}
 r=\deg(\Delta )\deg(R )=d(n+1)(n(n+1)d-d\mu) \ .
 \end{align}

 We summarize this discussion as follows .  Given a partition $\beta_{\bull}$ with $N$ parts let $\mathbb{E}_{\beta_{\bull}}$ denote the corresponding irreducible  $G$ module. 
   Let $X\ra\cpn$ be a linearly normal complex projective variety. We make the associations
\begin{align}
\begin{split}
& X  \rightarrow  R(X)^{\deg(\Delta)} \in \elam\setminus\{0\} \ , \ (n+1)\lambda_{\bull}= \big(\overbrace{ {r} , {r} ,\dots, {r}}^{n+1},\overbrace{0,\dots,0}^{N-n}\big) \ .\\
\ \\
& X \rightarrow  \Delta(X)^{\deg(R)} \in \emu\setminus\{0\} \ , \  n\mu_{\bull}= \big(\overbrace{ {r} , {r} ,\dots, {r}}^{n },\overbrace{0,\dots,0}^{N+1-n}\big) \ . 
\end{split}
\end{align}
 Moreover, the associations are $G$ equivariant:
\begin{align}
 R(\sigma\cdot X)=\sigma\cdot R(X) \ , \ \Delta(\sigma\cdot X)=\sigma\cdot \Delta(X) \ .
 \end{align}
The irreducible modules $\elam$ and $\emu$ admit the following descriptions
\begin{align}
\begin{split}
&\elam\cong  H^0(\mathbb{G}(N-n-1,N), \mathcal{O}\Big(\frac{r}{n+1}\Big))\cong \mathbb{C}_{r}[M_{(n+1)\times (N+1)}]^{SL(n+1,\mathbb{C})} \\
\ \\
& \emu \cong H^0(\mathbb{G}(N-n ,N), \mathcal{O}\Big(\frac{r}{n}\Big))\cong \mathbb{C}_{r}[M_{n\times (N+1)}]^{SL(n,\mathbb{C})}\ . 
\end{split}
\end{align}
\begin{remark}
\emph{Observe that $r$ is divisible by both $n$ and $n+1$. Therefore $\lambull$ and $\mubull$ are actual partitions.}
\end{remark}
Once more we give the definition of a (semi)stable\footnote{This is really K-stability as it \emph{should} have been defined.} projective variety.\newline

\noindent{\textbf{Definition.}}  Let $X\ra\cpn$ be an irreducible, $n$-dimensional, linearly normal complex projective variety. Then $X$ is \emph{\textbf {(semi)stable}}  if and only if 
\begin{align}
(R^{\ \deg(\Delta)}\ , \ \Delta^{\deg(R)}) 
\end{align}
 is a (semi)stable pair for the action of $G$.

%%%%%%%%%%%%%%%%%%%%%%%%%%%%%%%%%%%%%%%%%%%%%%%%%%%%%%%%%%%%%%%%%%%%%%%%%%%%%%%%%%%%%%%%%% 
 \section{Equivariant Extensions of Rational Maps}\label{equivexts}
 In this section we follow the first few paragraphs of \cite{deconcini87}, our primary goal is to put into context our  notion(s) of semistability  (Definitions  \ref{semistable}, \ref{stable} and \ref{dominate} ) and to give a complete proof of the indispensable \emph{numerical criterion} for pairs (see Theorem \ref{summary}) .

 To begin, let $G$ be an algebraic group. $H\leq G$ a Zariski closed (possibly finite) subgroup. Let $\mathcal{O}$ denote the algebraic homogeneous space $G/H$. The definition of semistable pair arises immediately upon studying equivariant completions of the space $\mathcal{O}$.  
\begin{definition}
\emph{An \textbf{\emph{embedding}} of $\mathcal{O}$ is a $G$ variety $X$ together with a $G$-equivariant embedding $i:\mathcal{O}\ra X$ such that $i(\mathcal{O})$ is an open dense orbit of $X$.}
\end{definition}
  $[\sigma]=\sigma H$ denotes the associated $H$ coset for any $\sigma\in G$. Then an embedding of $\mathcal{O}$ has a natural basepoint given by $o:=i([e])$ and we have that
\begin{align}
i(\mathcal{O})=G\cdot o \ ,\ \overline{G\cdot o}=X\ .
\end{align}
Let $(X_1,i_1)$ and $(X_2,i_2)$ be two embeddings of $\mathcal{O}$.  We recall the following well established notion.
\begin{definition}
\emph {A \textbf{\emph{morphism}} $\varphi$ from $(X_1,i_1)$ to $(X_2,i_2)$ is a $G$ equivariant regular map \newline $\varphi:X_1\ra X_2$ such that the diagram
\[
 \xymatrix{
  &X_1 \ar[dd]^{\varphi}  \\
 \mathcal{O}\ar[ru]^{i_1}\ar[rd]_{i_2}&\\
 & X_2  }
 \]
commutes. If a morphism  $\varphi$ exists we write $(X_1,i_1)\succsim (X_2,i_2)$ and we say that $(X_1,i_1)$ \textbf{\emph{dominates}} $(X_2,i_2)$. }
\end{definition}
\begin{remark}\emph{Observe that if a morphism  exists it is unique.}\end{remark}

Let $(X_1,i_1)$ and $(X_2,i_2)$ be two embeddings of $\mathcal{O}$ such that $(X_1,i_1)\succsim (X_2,i_2)$. Assume that these embeddings are both projective (hence complete) with very ample linearizations
\begin{align}
\mathbb{L}_1\in \mbox{Pic}(X_1)^G\ , \ \mathbb{L}_2\in \mbox{Pic}(X_2)^G 
\end{align}
 satisfying
\begin{align}
\varphi^*(\mathbb{L}_2)\cong \mathbb{L}_1 \ .
\end{align}
This is essentially Definition 1.2.1 of \cite{alexeev-brion2006}. Observe that the induced map of $G$ modules
\begin{align}
\varphi^*:H^0(X_2,\mathbb{L}_2) \ra H^0(X_1,\mathbb{L}_1) 
\end{align}
is injective, hence its dual map
\begin{align}
(\varphi^*)^t:H^0(X_1,\mathbb{L}_1)^{\vee}\ra H^0(X_2,\mathbb{L}_2)^{\vee}
\end{align}
is surjective and gives a rational map on the projectivizations of these spaces. The whole set up may be pictured as follows
\[
 \xymatrix{
 & X_1 \ar[dd]^{\varphi}\ar@{^{(}->}[r]^>>>{f_1 }&\  \mathbb{P}(H^0(X_1,\mathbb{L}_1)^{\vee})\ar@{-->}[dd]^{(\varphi^*)^t} \\
 \mathcal{O}\ar[ru]^{i_1}\ar[rd]_{i_2}& &\\
 & X_2 \ar@{^{(}->}[r]^>>>{f_2}& \mathbb{P}(H^0(X_2,\mathbb{L}_2)^{\vee})}
 \]
We isolate some features of this situation.
\begin{enumerate}
\item There are $u_i \in H^0(X_i,\mathbb{L}_i)^{\vee}\setminus\{0\}$ such that $X_i\cong f_i(X_i)=\overline{G\cdot [u_i]}$ $(i=1,2)$ . \\
\ \\
\item $(\varphi^*)^t(u_1)=u_2$ \ . \\
\ \\
\item $\mbox{Span}(G\cdot u_i )=H^0(X_i,\mathbb{L}_i)^{\vee}$ \ . \\
\ \\
\item The map $(\varphi^*)^t:G\cdot [u_1]\ra G\cdot [u_2]$ extends to a regular map between the Zariski closures of these orbits.
\end{enumerate}

We abstract (1)-(4) as follows.
Let $G$ be a complex reductive linear algebraic group.  We consider pairs $(\mathbb{E}; u)$ such that the linear span of $G\cdot u$ coincides with $\mathbb{E}$. Recall from the introduction that for any vector space $\mathbb{E}$ and any $u\in \mathbb{E}\setminus\{0\}$ we  let $[u]\in\mathbb{P}(\mathbb{E})$ denote the line through $u$. If $\mathbb{E}$ is a $G$ module define $\mathcal{O}_{u}:=G\cdot [u]\subset \mathbb{P}(\mathbb{E})$ the projective orbit of $[u]$ . Recall that $\overline{\mathcal{O}}_{u}$ denotes the Zariski closure of this orbit.

All of the author's work on the problem of characterizing a lower bound on the K-energy map of  a polarized manifold revolves around the following definition. The author first learned of this notion from the short note \cite{smirnov2004} .  

\begin{definition}\label{dominate} \emph{ 
$(\mathbb{E}; u)$  \textbf{\emph{dominates}} $(\mathbb{W}; w)$   
 if and only if there exists $\pi\in Hom(\mathbb{E},\mathbb{W})^G$ such that
$ \pi(u)=w$ and the induced  rational map  
$\pi:\mathbb{P}(\mathbb{E}) \dashrightarrow  \mathbb{P}(\mathbb{W})$
restricts to a regular map}  
$\pi:\overline{\mathcal{O}}_{u}\ra \overline{\mathcal{O}}_{w} \ $ \emph{between the Zariski closures of the orbits.}
  \end{definition}
The situation may be pictured as follows.
\[
 \xymatrix{
\mathcal{O}_u\ar[r]^{i_u}\ar[d]^{\pi} &  \overline{\mathcal{O}}_u\ar[d]^{\pi}\ar@{^{(}->}[r] & \mathbb{P}(\mathbb{E})\ar@{-->}[d]^{\pi} \\
 \mathcal{O}_w\ar[r]^{i_w} &  \overline{\mathcal{O}}_w\ar@{^{(}->}[r] & \mathbb{P}(\mathbb{W})}
 \]
Observe that the restriction of the map $\pi$  to $\overline{\mathcal{O}}_u$ is regular if and only if the following holds
  \begin{align}\label{disjoint}
 \qquad  \overline{\mathcal{O}}_u\cap \mathbb{P}(\ker \pi)=\emptyset \ .
\end{align}
Moreover one sees that the index $[G_u\ : \ G_w]$ is finite. In many examples one simply has $\mathcal{O}_u=\mathcal{O}_w$ .

 When  $(\mathbb{E}; u)$ dominates $(\mathbb{W}; w)$ we shall write $(\mathbb{E}; u)\succsim (\mathbb{W}; w)$. In this circumstance we observe that   
\begin{align} 
 & \pi(\mathbb{E})=\mathbb{W} \ \mbox{and} \ \mathbb{E}=\ker(\pi)\oplus \mathbb{W} \ \mbox{ ($G$-module splitting) } \ .
\end{align}
Therefore we may identify $\pi$ with projection onto $\mathbb{W}$ and $u$ decomposes as follows
\begin{align}
u=(u_{\pi},w) \ , \ \ker(\pi)\ni u_{\pi}\neq 0 \ 
\end{align}
where $u_{\pi}$ is the projection of $u$ onto the kernel of $\pi$. 
Again the reader can easily check that $(\ref{disjoint})$ is equivalent to
\begin{align}\label{project}
 \qquad \overline{G\cdot[(u_{\pi},w)]}\cap\overline{G\cdot[(u_{\pi},0)]}=\emptyset \quad \mbox{( Zariski closure in  $\mathbb{P}(\ker(\pi)\oplus\mathbb{W}$ ) )} \ .
\end{align}
Now we simply let
\begin{align}
  \mathbb{V}:=\ker{\pi} \ , \ v:=u_{\pi} 
  \end{align}
and we reformulate (\ref{project}) as follows. 
Given $(v\in \mathbb{V}\setminus \{0\}\ ;\ w\in \mathbb{W}\setminus \{0\})$ we consider the projective orbits
\begin{align}
\mathcal{O}_{vw}:=G\cdot[(v,w)]\subset \mathbb{P}(\mathbb{V}\oplus\mathbb{W}) \ , \ \mathcal{O}_{v}:=G\cdot[(v,0)]\subset \mathbb{P}(\mathbb{V}\oplus\{0\})
\end{align}
The purpose of the preceding  discussion is to explain and motivate Definition \ref{semistable}. We repeat it here.\newline
 \noindent\textbf{Definition.}
  The pair $(v,w)$ is \textbf{\emph{semistable}} if and only if  $ \overline{\mathcal{O}}_{vw}\cap\overline{\mathcal{O}}_{v}=\emptyset $ .
    
 %%%%%%%%%%%%%%%%%%%%%%%%%%%%%%%%%%%%%%%%%%%%%%%%%%%%%%%%%%%%%%%%%%%%%%%%%%
  \subsection{Toric Morphisms and T-Semistability}\label{toricmorphs}
In this subsection we study the semistability / dominance  relation $(X_1,i_1)\succsim (X_2,i_2)$ in the special case $G\cong T\cong (\mathbb{C}^*)^n$, an algebraic torus.  For simplicity we assume that the open orbits coincide with $G$, in other words, we assume that the stabilizers are both trivial.
 
 \begin{definition} 
\emph{The pair $(v,w)$ is \emph{\textbf{$T$-semistable}} provided the torus orbit closures are disjoint }
   \begin{align}
   \overline{T\cdot [(v,w)]}\cap \overline{T\cdot [(v,0)]}=\emptyset \ .
   \end{align}
\end{definition}

One expects  that $T$-semistability  admits a description in terms of weight polytopes which generalizes the numerical criterion of Geometric Invariant Theory (see \cite{dolgachev} pg.137 Theorem 9.2). This is indeed the case. To begin the discussion let's denote our torus by $H$. Let $\chi\in M_{\mathbb{Z}}$ be an $H$ character and $u\in N_{\mathbb{Z}}$ an algebraic one parameter subgroup satisfying $<\chi,u>=1$ . 
\[ 
\xymatrix{
1\ar[r]&T:=\mbox{Ker}(\chi)\ar@{^{(}->}[r]&H\ar[r]^{\chi}   & \mathbb{C}^*\ar@/_1pc/[l]\ar[r]&1}
\]
Let $\mathscr{A}, \mathscr{B} \subset M_{\mathbb{Z}}$ be (nonempty) finite subsets satisfying $a(u(\alpha))\equiv 1 , b(u(\alpha))\equiv 1$ for all $a\in \mathscr{A}$ and $b\in \mathscr{B}$ and all $\alpha \in \mathbb{C}^*$ . Define
\begin{align}
\mathscr{A}_+:=\{\chi\}\cup \{a+\chi\ |\ a\in \mathscr{A}\ \}\ , \ \mathscr{B}_+:=\{\chi\}\cup \{b+\chi\ |\ b\in \mathscr{B}\ \}\ .
\end{align}
There are two naturally associated $H$ representations $\mathbb{C}^{\mathscr{A}_+}$ and $\mathbb{C}^{\mathscr{B}_+}$ given by
\begin{align}
\begin{split}
&\mathbb{C}^{\mathscr{A}_+}=\mbox{span}\{\mathbf{e}_{\chi}\ ,\ \mathbf{e}_{\chi+a}\ | \ a\in \mathscr{A}\}\ , \  \mathbb{C}^{\mathscr{B}_+}=\mbox{span}\{\mathbf{f}_{\chi}\ ,\ \mathbf{f}_{\chi+b}\ | \ b\in \mathscr{B}\} \\
\ \\
&h\cdot (c_{\chi}\mathbf{e}_{\chi}+\sum_{a\in\mathscr{A}}c_a\mathbf{e}_{\chi+a})=\chi(h)(c_{\chi}\mathbf{e}_{\chi}+\sum_{a\in\mathscr{A}}a(h)c_a\mathbf{e}_{\chi+a})\\
\ \\
&h\cdot (c_{\chi}\mathbf{f}_{\chi}+\sum_{b\in\mathscr{B}}c_b\mathbf{f}_{\chi+b})=\chi(h)(c_{\chi}\mathbf{f}_{\chi}+\sum_{b\in\mathscr{B}}b(h)c_b\mathbf{f}_{\chi+b})\\
\ \\
& \mathbb{C}^{\mathscr{A}_+}\ni v:= \mathbf{e}_{\chi} +\sum_{a\in\mathscr{A}} 
\mathbf{e}_{\chi+a} \ ,\ \mathbb{C}^{\mathscr{B}_+}\ni w:= \mathbf{f}_{\chi} +\sum_{b\in\mathscr{B}} \mathbf{f}_{\chi+b}\ .
\end{split}
\end{align}
We define $T$ equivariant maps into projective spaces  in the usual manner:
\begin{align}
\begin{split}
&\varphi_{\mathscr{A}}:T\ra \mathbb{P}^{|\mathscr{A}|}:=\mathbb{P}(\mathbb{C}^{\mathscr{A}_+}) \ , \ \varphi_{\mathscr{A}}(t):=\big[ \mathbf{e}_{\chi} +\sum_{a\in\mathscr{A}}a(t)\mathbf{e}_{\chi+a}\big] \\
\ \\
&\varphi_{\mathscr{B}}:T\ra \mathbb{P}^{|\mathscr{B}|}:= \mathbb{P}(\mathbb{C}^{\mathscr{B}_+})\ , \ \varphi_{\mathscr{B}}(t):=\big[ \mathbf{f}_{\chi} +\sum_{b\in\mathscr{B}}b(t)\mathbf{f}_{\chi+b}\big] \\
\end{split}
\end{align}
Let $\pi\in Hom(\mathbb{C}^{\mathscr{A}_+},\mathbb{C}^{\mathscr{B}_+})^H$ be such that $\pi(v)=w$. Observe that these requirements force that the following conditions are met
\begin{align}
\mathscr{B}\subseteq\mathscr{A} \ \mbox{and}\ \ker(\pi)=\bigg\{\sum_{a\in \mathscr{A}\setminus\mathscr{B}}c_a\mathbf{e}_{\chi+a}\ |\ c_a\in \mathbb{C}\ \bigg\} \ .
\end{align}
Then we have exactly the same set up as before
\[
 \xymatrix{
 & X_{\mathscr{A}}:=\overline{\varphi_{\mathscr{A}}(T)}\ar[dd]^{\pi}\ar@{^{(}->}[r] & \mathbb{P}^{|\mathscr{A}|}\ar@{-->}[dd]^{\pi} \\
  {T}\ar[ru]^{\varphi_{\mathscr{A}}}\ar[rd]_{\varphi_{\mathscr{B}}}& &\\
 &  X_{\mathscr{B}}:=\overline{\varphi_{\mathscr{B}}(T)}\ar@{^{(}->}[r] & \mathbb{P}^{|\mathscr{B}|}}
 \]
 
 Recall from our previous discussion that the map $\pi$ extends to $X_{\mathscr{A}}\setminus \varphi_{\mathscr{A}}(T)$ if and only if $(X_{\mathscr{A}}\setminus \varphi_{\mathscr{A}}(T))\cap \mathbb{P}(\mbox{ker}(\pi))=\emptyset $. To test for this latter condition we study 
\begin{align}
\lambda^u(0)\cdot [v]=\varphi_{\mathscr{A}}(\lambda^u(0))\quad \mbox{for all $u\in N_{\mathbb{Z}}$} \ ,
\end{align}
where $\lambda^u$ is the one parameter subgroup corresponding to $u$.
\begin{align}
\lambda^u(t)\cdot v= \sum_{a\in\mathscr{A}\setminus\mathscr{B}}t^{(u,a)}\mathbf{e}_{\chi+a}+\big(\mathbf{e}_{\chi}+\sum_{b\in \mathscr{B}}t^{(u,b)}\mathbf{e}_{\chi+b}\big) \ .
\end{align}
 It follows at once that $\varphi_{\mathscr{A}}(\lambda^u(0))\in \mathbb{P}(\mbox{ker}(\pi))$ if and only if
 \begin{align}
 \min_{a\in\mathscr{A}\setminus\mathscr{B}}(u,a) <\ \min\{0,\min_{b\in\mathscr{B}}(u,b) \} \ .
 \end{align}
Therefore a necessary condition that $\pi$ extend to the closure of the torus orbit is the following
\begin{align}
\min\{0,\min_{b\in\mathscr{B}}(u,b) \}\leq \min_{a\in\mathscr{A}\setminus\mathscr{B}}(u,a) \ \mbox{for all $u\in N_{\mathbb{Z}}$} \ .  
\end{align}
 In fact, this necessary condition is also sufficient and can be formulated as follows.
\begin{theorem}\label{polytope}
\emph{The map $\pi$ extends to $X_{\mathscr{A}}$ if and only if $$\mbox{\textbf{conv}}(\mathscr{A}\setminus\mathscr{B})\subseteq \mbox{\textbf{conv}}(\{0\}\cup \mathscr{B}) .$$}
\end{theorem}
 
 The heart of the matter is to prove the following proposition which is closely related to the \emph{Orbit-Cone} correspondence of Toric Geometry (see \cite{coxbook}) .
\begin{proposition}\label{orbitcone}
\emph{Let $[y]\in X_{\mathscr{A}}\setminus \varphi_{\mathscr{A}}(T)$. Then there exists a $u\in N_{\mathbb{Z}}$ and $\tau\in T$ such that }
\begin{align}
\varphi_{\mathscr{A}}(\lambda^u(0))=\tau\cdot[y]\ .  
\end{align}
\end{proposition}
Before we begin the proof we fix some notation and make some preliminary remarks. Let $\mathbb{E}$ be a finite dimensional rational representation of $T$. As before we let $\mathscr{A}$ denote the weights of the representation. We have the eigenspace decompostion
\begin{align}
\mathbb{E}\cong \bigoplus_{a\in \mathscr{A}}\mathbb{E}_a \ .
 \end{align}
We let $m(a)$ denote the multiplicity of the weight $a$
\begin{align}
m(a):=\dim(\mathbb{E}_a) \ .
\end{align}
Without loss of generality we may assume that
\begin{align}
m(a)=1 \ \quad \mbox{for all $a\in \mathscr{A}$ .}
\end{align}
This amounts to the fact that we may assume that $\mathbb{E}$ is isomorphic to
\begin{align}
\mathbb{C}^{\mathscr{A}}:= \mbox{Map}(\mathscr{A}\ , \ \mathbb{C}) \ .
\end{align}
The action is given by
\begin{align}
(\tau\cdot f)(a)=a(\tau)f(a) \quad \mbox{for all $a\in \mathscr{A}$ and $f\in \mathbb{C}^{\mathscr{A}}$ .}
\end{align}
We let $\mathbf{e}_a$ denote the natural basis of $\mathbb{C}^{\mathscr{A}}$
\begin{align}
\mathbf{e}_a(b)= \delta_{ab} \ .
\end{align}
Therefore we may write
\begin{align}
f=\sum_{a\in \mathscr{A}}f(a)\mathbf{e}_a \quad \mbox{for all $f\in \mathbb{C}^{\mathscr{A}}$ .}
\end{align}
Proposition \ref{orbitcone} follows easily from its affine cousin.
\begin{proposition}\label{richardsontori}
\emph{Let $f \in \mathbb{C}^{\mathscr{A}}$. Given any boundary point
\begin{align}
f_{\infty}\in \overline{T\cdot f}\setminus T\cdot f 
\end{align}
there exists a $\tau\in T$ and a $u\in N_{\mathbb{Z}}$ such that
\begin{align}
\lim_{\alpha\ra 0}\lambda^u(\alpha)\cdot f=\tau\cdot f_{\infty}
\end{align}
}
\end{proposition}
The argument is really due to Richardson (see \cite{birkes71}) . We take the opportunity here to provide complete details.  
\begin{proof}
We may assume that $\mbox{supp}(f)=\mathscr{A}$. Let $\mathscr{B}\subset \mathscr{A}$ denote the support of $f_{\infty}$. We write
\begin{align}
f=\sum_{a\in \mathscr{A}\setminus \mathscr{B}}f(a)\mathbf{e}_a+\sum_{b\in \mathscr{B}}f(b)\mathbf{e}_b \ .
\end{align}
Our assumption that $f_{\infty}$ lies in the boundary is equivalent to the existence of a sequence $\tau_j\in T$ satisfying
\begin{align}\label{limit}
\begin{split}
& a(\tau_j)\ra 0 \quad \mbox{for all $a\in \mathscr{A}\setminus \mathscr{B}$ and }\\
\ \\
& b(\tau_j)f(b)\ra f_{\infty}(b)  \ .
\end{split}
\end{align}
Let $ {\mathbb{R}}\mathscr{B}\subset M_{\mathbb{R}}$ denote the real subspace generated by the elements of $\mathscr{B}$. We consider the quotient space
\begin{align}
W:= M_{\mathbb{R}}/  {\mathbb{R}}\mathscr{B}\ .
\end{align}
We consider the projection
\begin{align}
\pi: M_{\mathbb{R}}\ra W \ .
\end{align}
Let $\Delta$ be the convex polytope ( in $W$ ) :
\begin{align}
\Delta := \mbox{convexhull}\{\pi(a)\ |\ a\in \mathscr{A}\setminus \mathscr{B}\} \ .
\end{align}
\begin{claim}
$0\notin \Delta \ $ .
\end{claim}
If the claim is false, then there are real constants $\{r_a\}$ and $\{c_b\}$ where $a\in\mathscr{A}\setminus \mathscr{B}$ and $b\in \mathscr{B}$ such that
\begin{align}
r_a\geq 0 \ ,\ \sum_{a\in \mathscr{A}\setminus \mathscr{B}} r_a=1 
\end{align}
and
\begin{align}
\sum_{a\in \mathscr{A}\setminus \mathscr{B}}r_aa=\sum_{b\in \mathscr{B}}c_bb\ .
\end{align}
Hence for all $\tau\in T$ we have
\begin{align}\label{products}
\prod_{a\in \mathscr{A}\setminus \mathscr{B}}|a(\tau)|^{r_a}=\prod_{b\in\mathscr{B}}|b(\tau)|^{c_b} \ .
\end{align}
By substituting the sequence $\{\tau_j\}$ from (\ref{limit}) we derive a contradiction, as the left hand side of  (\ref{products}) tends to zero but the right hand side does not. This completes the proof of the claim.

Therefore (by the Hyperplane Separation Theorem) there is a linear functional
\begin{align}
l:W\ra\mathbb{R}
\end{align}
such that
\begin{align}\label{positive}
l(\pi(a))>0 \ \mbox{for all $a\in \mathscr{A}\setminus \mathscr{B}$ .}
\end{align}
Next let $L_1,L_2,\dots, L_N$ be a $\mathbb{Z}$ basis of $M_{\mathbb{Z}}$ such that $\pi(L_1) , \dots , \pi(L_k)$ is a basis of $W$. Let $\theta_1,\dots,\theta_k$ denote the dual basis. We define the \emph{rational dual} to $W$ by
\begin{align}
W^{\vee}_{\mathbb{Q}}=\mathbb{Q}\theta_1\oplus \mathbb{Q}\theta_2\oplus\dots \oplus \mathbb{Q}\theta_k \ .
\end{align}
By density of $\mathbb{Q}$ we have the following.
\begin{claim} \emph{There is a $g\in W^{\vee}_{\mathbb{Q}}$ such that
\begin{align}
g(\pi(a))>0 \ \mbox{for all $a\in \mathscr{A}\setminus \mathscr{B}$ .}
\end{align}
}
\end{claim}
Next we look at the composition
\begin{align}
M_{\mathbb{R}}\xrightarrow{\pi}W\xrightarrow{g}\mathbb{R} \ .
\end{align}

\begin{claim}
\emph{ $g\circ \pi$ is a \emph{rational} linear functional on $M_{\mathbb{R}}$. Precisely
\begin{align}
g\circ \pi(L_i)\in \mathbb{Q} \quad \mbox{for all $1\leq i\leq N$ .}
\end{align}
}
\end{claim}
We already know that $g\circ \pi(L_i)\in \mathbb{Q}$ when $1\leq i\leq k$. To get the remaining basis elements we observe that
the exact sequence
\begin{align}
0\ra  \mathbb{R}\mathscr{B}\ra M_{\mathbb{R}}\xrightarrow{\pi}W\ra 0
\end{align}
shows that we may find a basis of $M_{\mathbb{R}}$ of the shape
\begin{align}
\{b_1,\dots,b_{N-k}; L_1,\dots,L_k\}\quad \mbox{where $b_i\in\mathscr{B}\subset M_{\mathbb{Z}}$ .}
\end{align}
Let $j\in\{k+1,\dots, N\}$. By Cramer's rule there are rational numbers $q_{ij}$ such that
\begin{align}\label{cramer}
L_j=\sum_{i=1}^{k}q_{ij}L_i+\sum_{i=k+1}^{N}q_{ij}b_{i-k} \ .
\end{align}
Applying $g\circ \pi$ to both sides of (\ref{cramer}) proves the claim.

Therefore, there is a positive integer $m$ satisfying
\begin{align}
m (g\circ \pi)\in N_{\mathbb{Z}} \ .
\end{align}
Finally we define $u:=m (g\circ \pi)$ and we observe that the corresponding one parameter subgroup $\lambda^u$ satisfies
\begin{align}
\lambda^u(0)f:=\lim_{\alpha\ra 0}\lambda^u(\alpha)\cdot f= \sum_{b\in\mathscr{B}}f(b)\mathbf{e}_b  \ .
\end{align}
Recall from (\ref{limit}) that 
\begin{align}
b(\tau_i)f(b)\ra f_{\infty}(b) \ .
\end{align}
Therefore
\begin{align}
\overline{T\cdot f_{\infty}}\subset \overline{T\cdot \lambda^u(0)f} \ .
\end{align}
Since 
\begin{align}
 \mbox{supp}(f_{\infty}) =\mbox{supp}(\lambda^u(0)f)=\mathscr{B}
 \end{align} 
  these two toric\footnote{Normality is not essential here.} varieties \emph{have the same dimension\footnote{The (common) dimension is the rank of the lattice $\mathbb{Z}\mathscr{B}$ .}}. Since these varieties are irreducible \emph{they coincide}
\begin{align}
\overline{T\cdot f_{\infty}}= \overline{T\cdot \lambda^u(0)f} \ .
\end{align}
Since any two (nonempty) Zariski open subsets of an irreducible variety must meet we see that
\begin{align}
T\cdot f_{\infty}\cap T\cdot \lambda^u(0)f \neq \emptyset \ .
\end{align}
Since these sets are orbits we must have
\begin{align}
T\cdot f_{\infty}= T\cdot \lambda^u(0)f  \ .
\end{align}
This completes the proof of Proposition \ref{richardsontori} .
\end{proof}
 
 Now we come to the core result of the theory of semistable pairs. It is very closely related to, and functions in exactly the same way as, the ``one variable'' numerical criterion.
\begin{theorem}\label{numericalcriterion}
\emph{Let $\mathbb{V}$ and $\mathbb{W}$ be finite dimensional complex rational representations of $G$. Let $v$ and $w$ be (nonzero) elements of $\mathbb{V}$ and $\mathbb{W}$ respectively. Then the pair $(v,w)$ is semistable if and only if it is $T$ semistable for all maximal algebraic tori $T\leq G$.}
\end{theorem}
Theorem \ref{numericalcriterion} follows at once from the next proposition.
\begin{proposition}\label{subspace}\emph{Let $G$ be a classical group. Let $\mathbb{E}$ be a finite dimensional rational representation of $G$. Let $L\subset \mathbb{P}(\mathbb{E})$ be a proper $G$ invariant linear subspace. Let $u\in \mathbb{E}\setminus\{0\}$. Suppose that
\begin{align}
\overline{\mathcal{O}}_u\cap L\neq \emptyset\ .
\end{align}
Then there is a maximal algebraic torus $T\leq G$ such that
\begin{align}
\overline{T\cdot [u]}\cap L\neq \emptyset \ .
\end{align}
}
\end{proposition}
We record an important corollary of Propositions \ref{subspace} and \ref{richardsontori}.
\begin{corollary}\label{lambda}\emph{Under the preceding hypotheses, there exists a one parameter subgroup $\lambda:\mathbb{C}^*\ra G$ satisfying}
\begin{align}
\lim_{\alpha\ra 0}\lambda(\alpha)\cdot [u]\in L \ .
\end{align}
\end{corollary}
The argument is an adaption of one originally due to R. Richardson (see \cite{birkes71} pgs. 464-465). 
\begin{proof} The proof is by contradiction. Therefore we assume that  
\begin{align}
L\cap \big(\overline{\mathcal{O}}_u\setminus {\mathcal{O}}_u\big) \neq \emptyset
\end{align}
and for \emph{every} algebraic torus $H$ in $G$ we have 
\begin{align}
\overline{H\cdot[u]}\cap L=\emptyset \ .
\end{align}

Fix any maximal algebraic torus $T$ of $G$. There is a finite collection 
\begin{align}
\mathscr{C}=\{ U_1 , U_2, \dots , U_m\}
\end{align}
of $T$ invariant affine open sets of  $\mathbb{P}(\mathbb{E})$ satisfying
\begin{enumerate}
 \item $L\cap U_i\neq \emptyset$ for all $i\in\{1,2,\dots,m\} $ .\\
 \ \\
 \item $L\subset \cup_{i=1}^mU_i $ . \\
 \ \\
 \item For all $\kappa\in K$ (a maximal compact of $G$) there is an $i=i(\kappa)$ such that\\
 \begin{align}
 T\kappa\cdot [u]\subset U_{i(\kappa)} \ .
 \end{align}
 \end{enumerate}
 
 The last assumption implies that for every $\kappa\in K$ there is a ball $\kappa\in B_{\delta(\kappa),i(\kappa)}\subset K$ (in any Riemannian metric on $K$ ) such that 
 \begin{align}
 x\cdot[u]\in U_{i(\kappa)} \ \mbox{for all $x\in B_{\delta(\kappa),i(\kappa)}$}\ .
 \end{align}
Furthermore we have that
\begin{align}
 \overline{T\kappa\cdot [u]} \cap L\cap U_{i(\kappa)}=\emptyset \ .
 \end{align}
 Therefore there exists a $T$ invariant
 \begin{align}
 f=f_{\kappa , i(\kappa)}\in \mathbb{C}[U_{i(\kappa)}]^T \ 
 \end{align}
 satisfying
 \begin{enumerate}
 \item $ f|_{ \overline{T\kappa\cdot [w]}\cap U_{i(\kappa)}}\equiv 1$ . \\
 \ \\
\item $f|_{L\cap U_{i(\kappa)}}\equiv 0 $ .
\end{enumerate}
  
 Therefore by choosing $r(\kappa)<\delta(\kappa)$ we may assume that
 \begin{align}
 |f_{\kappa , i(\kappa)}(x\cdot[u])|>0 \ \mbox{for all $x\in B_{r(\kappa),i(\kappa)}$} \ .
 \end{align}
 Now we extract a finite covering  $\{B_{ij}\}$ of $K$ by balls $B_{ij}$ satisfying
 \begin{enumerate}
 \item $x\cdot [u]\in U_{j}$  for all $x\in B_{ij}$ . \\
 \ \\
 \item $|f_{ij}(x\cdot [u])|>0$ for all $x\in B_{ij}$ .
 \end{enumerate}
 Let $\{\varphi_{ij}\}$ be a partition of unity subordinate to this cover. We define a continuous function $F$ on $K$ by
 \begin{align}
 F(x):= \sum_{i,j}\varphi_{ij}(x)|f_{ij}(x\cdot [u])| \ .
 \end{align}

 It is clear from the construction that $F$ is positive. Therefore by compactness $F$ has a positive lower bound on $K$.
 Our assumption at the outset was 
 \begin{align}
 \overline{\mathcal{O}}_u\cap L\neq \emptyset \ .
 \end{align}
 The Cartan decomposition
 \begin{align}
 G=KTK
 \end{align}
 and the fact that $L$ is closed and $G$-invariant imply that there is a sequence
 \begin{align}
\lim_{l\ra \infty} t_l\kappa_l\cdot[u]\in L\quad t_l\in T \ , \ \kappa_l\in K .
\end{align}
Then by $T$-invariance of the $f_{ij}$ we have
\begin{align}
\begin{split}
 F(\kappa_l)&=\sum_{i,j}\varphi_{ij}(\kappa_l)|f_{ij}(\kappa_l\cdot [u])| \\
\ \\
&=\sum_{i,j}\varphi_{ij}(\kappa_l)|f_{ij}(t_l\kappa_l\cdot [u])| \ .
\end{split}
\end{align}
Since the $f_{ij}$ vanish on $L\cap U_j$  we see that $F(\kappa_l)\ra 0$ as $l\ra \infty$. This contradicts the fact that $F$ has a positive lower bound on $K$ and we are done.
\end{proof}
We summarize the main results of this section as follows. For the ``weights'' $w_{\lambda}(w)$ with respect to a one parameter subgroup $\lambda$ of $G$ see Definition \ref{weight}.
\begin{theorem}\label{summary} 
 The following statements are equivalent.  
\emph{
\begin{enumerate}
\item   $(v,w)$ is {semistable}. \ \\
 \item $(v,w)$ is $T$-semistable for all $T\leq G$. \ \\
\item   $\mathcal{N}(v)\subset\mathcal{N}(w)$ for all maximal tori $T\leq G$ .\ \\
 \item $w_{\lambda}(w)\leq w_{\lambda}(v)$ for all one parameter subgroups $\lambda$.\ \\
 \item For every maximal algebraic torus $T\leq G$ and $\chi\in \mathscr{A}_T(v)$ there exists \newline an integer $d >0$ and a {\emph{\textbf{relative invariant}}} $f \in \mathbb{C}_d[\ \mathbb{V}\oplus \mathbb{W}\ ]^{T}_{d\chi}$ such that
\begin{align}
\begin{split}
&f(v\ , \ w)\neq 0 \ \mbox{and} \ f|_{\mathbb{V}}\equiv 0 \ .
\end{split}
\end{align}
\end{enumerate}}
\end{theorem}
 We provide a simple application of the numerical criterion.
  \begin{example}
\emph{Let $\mathbb{V}_e$ and  $\mathbb{V}_d$ be irreducible $SL(2,\mathbb{C})$ modules with highest weights $e , d \in\mathbb{N}$. These are well known to be spaces of homogeneous polynomials in two variables. Let $f$ and $g$ be two such polynomials in $\mathbb{V}_e\setminus\{0\}$ and $\mathbb{W}_d\setminus\{0\}$
respectively. Theorem \ref{numericalcriterion} shows that the pair $(f,g)$ is semistable if and only if the following two conditions are satisfied
\begin{align}\label{d-e/2}
\begin{split}
& 1) \ e\leq d \ . \\
\ \\
& 2)\ \mbox{For all $p\in \mathbb{P}^1$}\ \mbox{ord}_p(g)-\mbox{ord}_p(f)\leq \frac{d-e}{2} \ .
\end{split}
\end{align}
This can be proved by factoring the polynomials.
In particular when $e=0$ and $f=1$ we see that $(1,g)$ is semistable if and only if 
\begin{align}
 \mbox{ord}_p(g) \leq \frac{d}{2} \ \mbox{for all $p\in \mathbb{P}^1$}\ .
\end{align}
 Assume that $e=d-1$. Suppose that $(f,g)$ is a semistable pair . Then by (\ref{d-e/2}) we get
\begin{align}\label{d=e}
\mbox{ord}_p(g)\leq \mbox{ord}_p(f) \quad \mbox{for all $p\in \mathbb{P}^1$} \ .
\end{align}
Let $\{p_1,p_2,\dots,p_d\}$ be the zeros of $g$ on $\mathbb{P}^1$ counted with multiplicity. Then by (\ref{d=e}) we have
\begin{align}
d=\sum_{1\leq i\leq d}\mbox{ord}_{p_i}(g)\leq \sum_{1\leq i\leq d}\mbox{ord}_{p_i}(f)\leq d-1 \ .
\end{align}
Therefore there is no  semistable pair in $\mathbb{V}_{d-1}\oplus\mathbb{V}_d$.
When $d=e$ the pair $(f,g)$ is semistable if and only if
\begin{align}
\mathbb{C}f=\mathbb{C}g \ . 
\end{align}}
\end{example}

%%%%%%%%%%%%%%%%%%%%%%%%%%%%%%%%%%%%%%%%%%%%%%%%%%%%%%%%%%%%%%%%%%%%%%%%%%%%%%%%%%%%%%%%%
 \subsection{A Kempf-Ness type functional}\label{kempfness} In this section we study semistability in terms of a Kempf-Ness type functional. In fact it is from this point of view that the author arrived at the definition of semistability. As always $(\mathbb{V}, v)$ and $(\mathbb{W}, w)$ are finite dimensional complex rational representations of $G$ together with a pair of nonzero vectors. We equip $\mathbb{V}$ and $\mathbb{W}$ with Hermitian norms .  We are interested in the function on $G$ which we call the \emph{energy of the pair} $(v,w)$ :
 \begin{align}\label{energy}
 G\ni \sigma\ra p_{vw}(\sigma):=\log||\sigma\cdot w||^2-\log||\sigma\cdot v||^2 \ .
 \end{align}
 Then we have the following fact.
 \begin{proposition}\label{vwlowerbound}
\emph{  $p_{vw}$ is bounded from below on $G$ if and only if $(v,w)$ is semistable.}
 \end{proposition}
 The proposition is a consequence of the following observation.
 \begin{lemma}\label{distance}
 \begin{align*}
 p_{vw}(\sigma)=\log\tan^2d_g(\sigma\cdot [(v,w)] , \sigma\cdot [(v,0)]) \ ,
 \end{align*}
 \emph{where $d_g$ denotes the distance in the Fubini-Study metric on} $\mathbb{P}(\mathbb{V}\oplus\mathbb{W})$ .
 \end{lemma}
\begin{proof}
 Let $u,v\in \mathbb{V}$ and let $(\cdot,\cdot)$ be any Hermitian inner product on $\mathbb{V}$ with associated Fubini-Study metric $g$ on $\mathbb{P}(\mathbb{V})$. Recall the distance formula
\begin{align*}
\cos d_g([u],[v])=\frac{|(u,v)|}{||u||||v||}\ .
\end{align*}
With the orthogonal direct sum Hermitian form on $\mathbb{V}\oplus\mathbb{W}$ we have for any $\sigma\in G$
\begin{align*}
\cos d_g(\sigma\cdot [(v,w)] , \sigma\cdot [(v,0)])&= \frac{|(\sigma\cdot [(v,w)] , \sigma\cdot [(v,0)])|}{\sqrt{||\sigma\cdot v||^2+||\sigma\cdot w||^2}||\sigma\cdot v||}\\
 \ \\
 &=\frac{||\sigma\cdot v||}{\sqrt{||\sigma\cdot v||^2+||\sigma\cdot w||^2}} \ .
\end{align*}
Since
\begin{align*}
\sec^2d_g(\sigma\cdot [(v,w)] , \sigma\cdot [(v,0)])&=1+\tan^2d_g(\sigma\cdot [(v,w)] , \sigma\cdot [(v,0)]) \ ,
\end{align*}
we may conclude that
\begin{align*} 
\tan^2d_g(\sigma\cdot [(v,w)] , \sigma\cdot [(v,0)])  = \frac{||\sigma\cdot w||^2}{||\sigma\cdot v||^2}\ .
\end{align*}
Now take the log of both sides.
\end{proof}
Observe that for any $\sigma,\tau \in G$ we have the inequality
\begin{align}\label{sigma-tau}
d_g(\sigma\cdot [(v,w)] , \sigma\cdot [(v,0)])\leq d_g(\sigma\cdot [(v,w)] , \tau\cdot [(v,0)])\ .
\end{align}
 As a corollary of (\ref{sigma-tau}) and Lemma \ref{distance} we have the much more refined version of Proposition \ref{vwlowerbound}.
\begin{corollary}\label{infi}\emph{ The infimum of the energy of the pair $(v,w)$ is as follows}
 \begin{align}\label{inf}
 \inf_{\sigma\in G}p_{vw}(\sigma)=\log\tan^2d_g(\overline{\mathcal{O}}_{vw},\overline{\mathcal{O}}_{v}) \ .  
 \end{align}  
 \end{corollary}
 Our whole approach to the standard conjectures in K\"ahler Geometry is based on this identity. 
 
Corollary \ref{infi} and Corollary \ref{lambda}  imply the following result.
 \begin{proposition} 
\emph{Assume that
\begin{align}
\inf_{\sigma\in G}p_{vw}(\sigma)=-\infty \ .
\end{align}
Then there exists an algebraic one parameter subgroup $\lambda$ of $G$ such that}
\begin{align}
\lim_{\alpha\ra 0}p_{vw}( {\lambda(\alpha)})= -\infty    \ .
\end{align}
\end{proposition}
\begin{definition}\label{properenergy}
\emph{$ p_{vw}$ is \textbf{\emph{proper}} if and only if there is an $m\in\mathbb{Z}_{>0}$ and a constant $B$ such that
\begin{align}
 p_{vw}(\sigma)+ \frac{1}{m+1}\log ||\sigma \cdot v||^2 \geq \frac{\deg(\mathbb{V})}{m+1}\log||\sigma||_{hs}^2 +B\ .
\end{align}
  $||\sigma||_{hs}$ denotes the Hilbert-Schmidt norm of $\sigma $ with respect to some Hermitian metric on the standard representation.}
\end{definition}  
\begin{theorem}
\emph{$(v,w)$ is a stable pair if and only if the energy $p_{vw}$ is proper.}
\end{theorem}
An application of Theorem \ref{summary} shows that there is a constant $C$ such that
\begin{align} 
   \deg(\mathbb{V}) \log||\sigma||_{hs}^2 \geq   \log ||\sigma \cdot v||^2 +C\ .
  \end{align}
\begin{remark}\label{exponent}
\emph{This inequality shows that if $(v,w)$ is stable with ``exponent $m$'', then it is ``$l$-stable'' for all $l\geq m$.}
\end{remark}
%%%%%%%%%%%%%%%%%%%%%%%%%%%%%%%%%%%%%%%%%%%%%%%%%%%%%%%%%%%%%%%%%%%%%%%%%%%%%%%%%%%%%%%%%
  \subsection{K-Energy Maps and (Semi)stable Pairs} Let $P$ be a numerical polynomial
 \begin{align}
 P(T)=c_n\binom{T}{n}+c_{n-1}\binom{T}{n-1}+\dots \ .
 \end{align}
 Let $\mathscr{H}^P_{\cpn}$denote the ``Hilbert Scheme'' of smooth, linearly normal subvarieties of $\cpn$ with Hilbert polynomial $P$.
 We require the following Theorem, recently obtained by the author.  
 \begin{theorem}\label{theorema}(Theorem A , Paul \cite{paul2012})
\emph{ There is a constant $M$ depending only on $ c_n , c_{n-1}$ and the Fubini Study metric such that for all points $[X]\in \mathscr{H}^P_{\cpn}$ and all $\sigma\in G$ we have
  \begin{align}
   \Bigg |\ d^2(n+1) \nu_{{\om_{FS}}|_{X}}(\vps)-  
  2  \log\frac { ||\sigma\cdot \Delta(X)^{\deg(R)}||}{||\sigma\cdot R(X)^{\deg(\Delta)}||} \ \Bigg| \leq M   \ .
 \end{align}
   $R$ and $\Delta$ have been scaled to unit length. }
 \end{theorem}
We now can give a proof of Theorem \ref{mainresult}. Curiously, the harder direction of the equivalence is
\begin{align}
\mbox{PROPER $\implies$ STABLE  }
\end{align}
This requires Theorem \ref{summary}. The other direction is easy.
Theorem \ref{theorema} and the assumption that $\nu_{\om}$ is proper on $\mathcal{B}$  give the inequality  
\begin{align}
    \log\frac { ||\sigma\cdot \Delta(X)^{\deg(R)}||^2}{||\sigma\cdot R(X)^{\deg(\Delta)}||^2}\geq Ad^2(n+1)J_{\om}(\vps)+C \ 
\end{align}
where $A>0$ and $C$ are constants . 
Recall from \cite{gacms} that
\begin{align}
\begin{split}
d^2(n+1)J_{\om}(\vps)&=d(n+1)\int_X\vps\om^n-d\log||\sigma\cdot R||^2 \\
\ \\
&=\frac{1}{n(n+1)-\mu}\Big(-\log||\sigma\cdot R^{\deg(\Delta)}||^2 + \deg(\elam)\int_X\vps\frac{\om^n}{d}\Big) \ .
\end{split}
\end{align}
Choose $m\in \mathbb{Z}>0$ so that
\begin{align}
\frac{A}{n(n+1)-\mu}\geq \frac{1}{m+1} \ .
\end{align}
Now we have
\begin{align}\label{inequality}
(m+1) \log\frac { ||\sigma\cdot \Delta(X)^{\deg(R)}||^2}{||\sigma\cdot R(X)^{\deg(\Delta)}||^2}\geq  \deg(\elam)\int_X\vps\frac{\om^n}{d}-\log||\sigma\cdot R^{\deg(\Delta)}||^2 +C \ .
\end{align}
The trick to deal with the mean is to use the numerical criterion. Let $\lambda$ be any one parameter subgroup of $G$. Then the inequality (\ref{inequality}) implies that
\begin{align}
(m+1)\big(w_{\lambda}(\Delta(X)^{\deg(R)})-w_{\lambda}(R(X)^{\deg(\Delta)})\big) \leq \deg(\elam)w_{\lambda}(\mathbb{I})-w_{\lambda}(R(X)^{\deg(\Delta)}) \ .
\end{align}
By Theorem \ref{summary} item $4)$ this is equivalent to the stability of $X\ra\cpn$ with exponent $m$.
%%%%%%%%%%%%%%%%%%%%%%%%%%%%%%%%%%%%%%%%%%%%%%%%%%%%%%%%%%%%%%%%%%%%%%%%%%%%%%%%%%%%%%%%%%
\subsection{The Classical Futaki Character}\label{futaki}
 As in the preceding sections  $G$ denotes a classical linear algebraic group. $\mathbb{V}$, $\mathbb{W}$ are finite dimensional complex rational representations of $G$. Let $v\in \mathbb{V}\setminus \{0\}$ and $w\in \mathbb{W}\setminus \{0\}$ . As usual $[v]$ denotes the corresponding point in the projective space $\mathbb{P}(\mathbb{V})$ and $G_{[v]}$ denotes the stabilizer of the line through $v$. Therefore there is a \emph{character}
\begin{align*}
\chi_{v}:G_{[v]}\ra \mathbb{C}^* \ , \ \sigma\cdot v= \chi_{v}(\sigma)\cdot v \ .
\end{align*}
\begin{definition} \emph{Let $v\in \mathbb{V}\setminus \{0\}$ and $w\in \mathbb{W}\setminus \{0\}$ . Then the \textbf{\emph{automorphism group}} of the pair $(v,w)$ is
the algebraic subgroup of $G$ given by 
\begin{align*}
\mbox{Aut}(v,w):=G_{[v]}\cap G_{[w]} \ .
\end{align*}}
\end{definition}
Let $\mathfrak{aut}(v,w)$ denote the Lie algebra of  $\mbox{Aut}^o(v,w)$ .
\begin{example}
\emph{Let $X\ra \cpn$ be a smooth subvariety of $\cpn$ satisfying our usual hypotheses. Then
\begin{align}
 \mbox{Aut}(X , \mathcal{O}(1)|_X)\cong  \mbox{Aut}(R(X),\Delta(X)) \ .
\end{align}
}
\end{example}
 \begin{definition} \emph {Let $\mathbb{V} , \mathbb{W}$ be finite dimensional complex rational representations of $G$. Let $v,w$ be two nonzero vectors in $\mathbb{V} , \mathbb{W}$ respectively. Then the \textbf{\emph{Futaki character}} of the pair $(v,w)$ is  the algebraic homomorphism
\begin{align}
F:=\chi_{w}\chi_{v}^{-1}: \mbox{Aut}(v,w) \ra \mathbb{C}^* \ 
\end{align}
induced by the one dimensional representation $\mathbb{C}w\otimes (\mathbb{C}v )^{\vee}$ . We set $ {F}_*$ to be the corresponding Lie algebra character
\begin{align*}
{F}_{*}:={d\chi_{w}}-{d\chi_{v}}:\mathfrak{aut}(v,w)\ra \mathbb{C} \ 
\end{align*}
where ${d\chi_{v}}$ denotes the differential of $\chi_{v}$ at the identity.}
\end{definition}
\begin{remark}
\emph{At this point the order is not important. That is, we could equally well consider $\chi_{w}^{-1}\chi_{v}$ .}
\end{remark}
Let $\tau\in G$ and $\sigma\in G_{[v]}\cap G_{[w]}$, then the diagram below is commutative.
\begin{align}
\xymatrix{
 \mathbb{C}w\otimes (\mathbb{C}v)^{\vee} \ar[d] ^{\chi_{w}\chi_{v}^{-1}(\sigma) } \ar[r]^{\alpha_{\tau}}&\mathbb{C}\tau\cdot w\otimes (\mathbb{C}\tau\cdot v)^{\vee} \ar[d]^{\chi_{\tau\cdot w}\chi_{\tau\cdot v}^{-1}(Ad_{\tau^{-1}}(\sigma))} \\
\mathbb{C}w\otimes (\mathbb{C}v)^{\vee} \ar[r] ^{\alpha_{\tau}}&\mathbb{C}\tau\cdot w\otimes (\mathbb{C}\tau\cdot v)^{\vee} 
 }  
 \end{align}
 This shows that the Futaki character only depends on the \emph{orbit} of the pair $(v,w)$ .

We can decompose the identity component of $\mbox{Aut}(v,w)$
\begin{align}
\mbox{Aut}^o(v,w)=S\rtimes U \ ,
\end{align} 
 $S$ is reductive and $U$ is unipotent . Then we have that
 \ \\
 \begin{center}\emph{ $F$ is completely determined  on $S$}.  \end{center}
 \ \\
 Let $T\leq S$ be any maximal algebraic torus in $S$. Then we have that 
 \ \\
\begin{center}\emph{ $F $ is completely determined  on $T$}.  \end{center}

Recall that
 \begin{align}
 N_{\mathbb{Z}}(T)\cong \{\mbox{algebraic one parameter subgroups $\lambda$ of $T$ }\} \ .
 \end{align}
Finally we have that
\begin{center}\emph{ $F $ is completely determined  on $N_{\mathbb{Z}}(T)$.}\end{center}
\begin{align}
F : N_{\mathbb{Z}}(T)\ra \mathbb{Z} \ , \ F (\lambda)=<\chi_{w},\lambda>- <\chi_{v},\lambda> \ .
\end{align}
In many cases one has that $\mbox{Aut}^o(v,w)$ is {trivial}. In such a situation we  introduce a generalization of $F$. In what follows $H$ denotes a maximal algebraic torus of $G$.
\begin{definition}\label{weight}
\emph{Let $\mathbb{V}$ be a rational representation of $G$. Let $\lambda$ be any degeneration in $H$  . The \textbf{\emph{weight}}  $w_{\lambda}(v)$  of $\lambda$ on $v\in \mathbb{V}\setminus\{0\}$ is the integer}
\begin{align}
w_{\lambda}(v):= \mbox{\emph{min}}_{ \{ x\in \mathcal{N}(v)\}}\ l_{\lambda}(x)= \mbox{\emph{min}} \{ <\chi,\lambda>| \chi \in \mbox{\emph{supp}}(v)\}\ .
\end{align}
 \noindent\emph{ Alternatively, $w_{\lambda}(v)$ is the unique integer such that}
\begin{align}
\lim_{|t|\rightarrow 0}t^{-w_{\lambda}(v)}\lambda(t)v \  \mbox{ {exists in $\mathbb{V}$ and is \textbf{not} zero}}.
\end{align}
\end{definition}
Let $\Delta(G)$ denote the space of algebraic one parameter subgroups of $G$.   
\begin{definition} \emph{The \textbf{\emph{generalized Futaki character}} of the pair $(v,w)$ is the map }
\begin{align}
F_{gen} :\Delta (G)\ra \mathbb{Z}\ , \ F_{gen} (\lambda):=w_{\lambda}(w)-w_{\lambda}(v) \ .
\end{align}
\end{definition}

The following is a restatement of Theorem \ref{numericalcriterion}. 
\begin{proposition}\emph{$F_{gen}(\lambda)\leq 0$ for all $\lambda \in \Delta(G)$ if and only if $(v,w)$ is  semistable.  }
\end{proposition}
 
The energy of the pair $(v,w)$ and the generalized character $F_{gen}$ are related as follows (see also \cite{paul2012} (2.30) pg. 269 ) .
 \begin{proposition}\label{fdkenergy}
\emph{ Let $\lambda\in\Delta(G)$ . Then there is an asymptotic expansion as $|t|\ra 0$
 \begin{align}
 p_{vw}(\lambda(t))=F_{gen}(\lambda)\log|t|^2+O(1) \ .  
 \end{align}}
\end{proposition} 
 In particular for $\sigma\in \mbox{Aut}^o(v,w)$ we have
 \begin{align}
 p_{vw}(\sigma) =\log|\chi_{w}(\sigma)|^2-\log|\chi_{v}(\sigma)|^2 \ . 
 \end{align}

We study the relationship between the generalized and classical Futaki invariants.  We have, as in the previous section, the Levi decomposition of $\mbox{Aut}^o(v,w)$
\begin{align}
\mbox{Aut}^o(v,w)=S\ltimes U \ ,
\end{align}
$S$ is reductive and $U$ is the unipotent radical of $\mbox{Aut}^o(v,w)$. Let $T \leq S$ be any maximal algebraic torus (possibly trivial). Since
$S\leq G$ there is a maximal algebraic torus $H$ in $G$ containing $T$. Fix any such $H$.   Then we have the short exact sequence of lattices
 \begin{align}
 0\ra L_{\mathbb{Z}}\xrightarrow{\iota}M_{\mathbb{Z}}(H)\xrightarrow{\pi_T}M_{\mathbb{Z}}(T)\ra 0 \ . 
   \end{align}
 
Recall that $\sigma\in \mbox{Aut}^o(v,w)$ acts on $w$ (resp. $v$)  via a character $\chi_w$ (resp. $\chi_v$ ) . We have the following.
\begin{proposition}\emph{ All the characters in the $H$ support of $w$ (or $v$) coincide upon restriction to $T$
\begin{align}
\begin{split}
&\chi\in\mbox{supp}(w)\Rightarrow \pi_T(\chi)=\chi_w \\
\ \\
&\eta\in\mbox{supp}(v)\Rightarrow \pi_T(\eta)=\chi_v \ .
\end{split}
\end{align}
Consequently, the difference of any two characters in supp($w$) (or supp($v$))  lies in $ {L_{\mathbb{Z}}}$ .  }
\end{proposition}
 
Extending scalars to $\mathbb{R}$ gives the sequence
\begin{align}
0\ra {L_{\mathbb{R}}}\xrightarrow{ {\iota}} M_{\mathbb{R}}(H)&\xrightarrow{ {\pi_T}} M_{\mathbb{R}}(T)\ra 0\ .
\end{align}

Then $\mathcal{N}(w)$ and $\mathcal{N}(v)$ both lay in affine subspaces of ${\mathbb{R}}^N$ modeled on $ {L_{\mathbb{R}}}$ . 
Now we suppose that 
\begin{align}
\mathcal{N}(v)\subseteq \mathcal{N}(w) \ .
\end{align}

Since ${ {\pi_T}}$ is {linear} we have that
\begin{align}
\{\chi_v\}={ {\pi_T}}\big(\mathcal{N}(v)\big)\subseteq {{\pi_T}}\big(\mathcal{N}(w)\big)=\{\chi_w\} \ .
\end{align}
We conclude 
\begin{align}
\chi_w=\chi_v \ . 
\end{align} 

We summarize the relationship between the character of the pair and $T$-semistability.
\begin{proposition}
\emph{ \begin{align}
 \begin{split}
&a)\ \mbox { $\mathcal{N}(v)$ and $\mathcal{N}(w)$ both lie in parallel affine subspaces of $M_{\mathbb{R}}$ modeled on $L_{\mathbb{R}}$ . }\\
&\mbox{Precisely, for any choice of $\chi\in \mbox{supp}(v)$ and $\eta\in \mbox{supp}(w)$ we have }\\
& \mathcal{N}(v)\subset \mathbb{A}_v:=\chi + L_{\mathbb{R}} \ \mbox{and}\ \mathcal{N}(w)\subset \mathbb{A}_w:=\eta + L_{\mathbb{R}} \ . \\
\ \\
&b) \ F_*\equiv 0 \ \mbox{if and only if} \ \mathbb{A}_v=\mathbb{A}_w \ . \\
\ \\
&c)\ F_*\not\equiv 0 \ \mbox{if and only if}\ \mathbb{A}_v\cap\mathbb{A}_w=\emptyset \ .\\
\ \\
& d)\ \mbox{ $F_*\equiv 0$  whenever the pair $(v,w)$ is semistable .}
 \end{split}
 \end{align}} 
 \end{proposition} 
 \begin{proposition}\label{finiteauto}
\emph{A stable pair $(v,w)$ has finite automorphism group.}
\end{proposition} 
\begin{proof}
Decompose the automorphism group
\begin{align}
\mbox{Aut}^o(v,w)=S\rtimes U \ ,
\end{align} 
where $S$ is reductive and $U$ is unipotent. Let $u\in U$. Recall that stability is equivalent to the coercive estimate
\begin{align}
 p_{vw}(\sigma)+ \frac{1}{m+1}\log ||\sigma \cdot v||^2 \geq \frac{\deg(\mathbb{V})}{m+1}\log||\sigma||_{hs}^2 +B\ .
\end{align}
Since $\chi_v$  and $\chi_w$ are trivial on $U$ we see that there is a constant $R$ (independent of $u$) such that
\begin{align}
 ||u||_{hs}^2\leq R \ .
\end{align}
Since a compact algebraic group is finite $U$ reduces to $\{1\}$. Now we deal with $S$.
Stability implies semistability, therefore the Futaki character $F$ is trivial. In particular for any degeneration $\lambda:\mathbb{C}^*\ra S$ we have the inequality
\begin{align}\label{weightineq}
<\chi_v,\lambda>\leq \deg(\mathbb{V})w_{\lambda}(\mathbb{I})\ .
\end{align}
Since the reverse inequality always holds we have \emph{equality} in (\ref{weightineq}). Observe that
\begin{align}
\deg(\mathbb{V})w_{\lambda^{-1}}(\mathbb{I})=<\chi_v,\lambda^{-1}>=-\deg(\mathbb{V})w_{\lambda}(\mathbb{I})
\end{align}
Therefore
\begin{align}
w_{\lambda^{-1}}(\mathbb{I})=-w_{\lambda}(\mathbb{I}) \ .
\end{align}
This forces $\lambda\equiv 1$. The conclusion is that $S$ contains no nontrivial algebraic tori. This forces $S=\{1\}$.
\end{proof}
 \begin{remark}
 \emph{In the conclusion of the proof we used that $G$ sits inside $SL$. If instead we let $G\leq GL$ then we find that $\lambda$ lies in the \emph{center}. }
 \end{remark}
  %%%%%%%%%%%%%%%%%%%%%%%%%%%%%%%%%%%%%%%%%%%%%%%%%%%%%%%%%%%%%%%%%%%%%%%%%%%%%%%%%%%%%%%
%%%%%%%%%%%%%%%%%%%%%%%%%%%%%%%%%%%%%%%%%%%%%%%%%%%%%%%%%%%%%%%%%%%%%%%%%%%%%%%%%%%%%%%%%%
\section{Appendix I : Background from K\"ahler Geometry}\label{append1}
In this paper we consider closed K\"ahler manifolds
\begin{align}  
(X^n,\om)  \ , \ n=\dim_{\mathbb{C}}(X) \ .
\end{align}
Recall that the K\"ahler form $\om$ is given locally by a Hermitian positive definite matirx of functions
  \begin{align}
  \om=\frac{\sqrt{-1}}{2\pi}\sum_{ i,j }g_{i\overline{j}}dz_{i}\wedge d\overline{z}_j  \ .
  \end{align}
The Ricci form of $\om$ is the smooth $(1,1)$ form on $X$ given by
\begin{align}
\begin{split}
 Ric(\om)&:= \frac{1}{2\pi\sqrt{-1}}\dl\dlb\log\det(g_{i\overline{j}}) \\
 \ \\
 &=\sum_{i,j}\frac{1}{2\pi\sqrt{-1}}R_{i\overline{j}}dz_{i}\wedge d\overline{z}_j \in \Gamma(\Lambda^{1,1}_X)
\end{split}
\end{align}
The scalar curvature is then the contraction of the Ricci curvature 
 \begin{align}
S(\om): =\sum_{i,j}g^{i\overline{j}}R_{i\overline{j}}\in C^{\infty}(X)\ .
 \end{align}  
\begin{align}
 V=\int_X\om^n \ , \ \mu=\frac{1}{V}\int_XS(\om)\om^n   \ .
 \end{align}
The space of K\"ahler metrics in the class $[\om]$ is given by
\begin{align}
\begin{split}
& \mathcal{H}_{\om}:=\{\vp\in C^{\infty}(X)\ |\ \om_{\vp}:=\om+\frac{\sqrt{-1}}{2\pi}\dl\dlb\vp  >0 \}  \ .
\end{split}
\end{align} 
  \begin{definition} (Mabuchi \cite{mabuchi}) 
 \emph{The \textbf{\emph{K-energy map}} $\nu_{\om}:\mathcal{H}_{\om}\ra \mathbb{R}$ is given by
\begin{align*}
\nu_{\om}(\varphi):=-\frac{1}{V}\int_0^1\int_{X}\dot{\varphi_{t}}(S(\om_{\varphi_{t}})-\mu)\omega_{t}^{n}dt 
\end{align*}
 $\vp_{t}$ is a $C^1$ path in $\mathcal{H}_{\om}$ satisfying $\vp_0=0$ , $\vp_1=\vp$ .}
 \end{definition}
\begin{remark} 
\emph{$\vp$ is a critical point for $\nu_{\om}$ iff  $ S(\om_{\varphi})\equiv \mu $ (a constant).}
\end{remark}

\begin{theorem}(Bando-Mabuchi \cite{bando-mabuchi87}) \emph{Assume that $X$ admits a K\"ahler Einstein metric. Then the K-energy map is bounded below on $\mathcal{H}_{\om}$.}
\end{theorem}

\begin{definition}\label{proper}(Tian \cite{tian97})
\emph{Let $(X,\om) $ be a K\"ahler manifold . The Mabuchi energy is \emph{\textbf{proper}} provided there exists constants $A>0$ and $B$ such that for all $\varphi\in \mathcal{H}_\om$ we have }
\begin{align}
\begin{split}
&\nu_{\om}(\vp)\geq A J_{\omega}(\varphi)+B  \\
 \ \\
 &  J_{\omega}(\varphi):= \frac{1}{V}\int_{X}\sum_{i=0}^{n-1}\frac{\sqrt{-1}}{2\pi}\frac{i+1}{n+1}\dl\varphi \wedge \dlb
\varphi\wedge \omega^{i}\wedge {\omega_{\varphi} }^{n-i-1}\ .
 \end{split}
\end{align}
\end{definition}

A deep analytic result in K\"ahler Einstein theory is the following theorem of Tian \footnote{One also needs the refinement from \cite{pssw2008} .}.

\begin{theorem}(Tian \cite{tian97}) \emph{Let $(X,\om) $ be a Fano manifold with $[\om]=C_1(X)$. Assume that $\mbox{Aut}(X)$ is finite. Then $X$ admits a K\"ahler Einstein metric if and only if the Mabuchi energy is proper.}
 \end{theorem}
 
  In this paper we are interested in the case where $[\om]$ is an \emph{integral} class, i.e. there is an ample line bundle $\mathbb{L}$ on $X$ and a Hermitian metric $h$ on $\mathbb{L}$ such that

 \begin{center} $ -\frac{\sqrt{-1}}{2\pi}\dl\dlb\log(h )=\om$ \ .\end{center}

In this situation Yau and Tian have conjectured that the existence of KE metrics should be related to a certain ``stability condition '' in the spirit of geometric invariant theory. In \cite{tian97}, Tian introduced the concept of K-stability and showed how it relates to the critical points of the Mabuchi energy through the projective models of $X$ furnished by high powers of the line bundle 
 \begin{align}
 X \xrightarrow{ }\mathbb{P}(H^0(X, \mathbb{L}^k)^{\vee}) \ .
\end{align}
To make this more precise one proceeds as follows. Given a basis $B:=< S_0,S_1,\dots, S_{N_k}>$  of $H^0(X, \mathbb{L}^k)$  it is not hard to see that
\begin{align}
\frac{1}{k}\rho_{k;B}:= \frac{1}{k}\log\Big(\sum_j|S_j|_{h^k}^2\Big) \in \mathcal{H}_{\om} \ .
\end{align}
 
Define the Bergman space ``at level $k$''
\begin{align}
\mathcal{B}_k:=\{\  \rho_{k,B}\ |\ \mbox{  $B$ is a basis of  $H^0(X, \mathbb{L}^k)$}\} \ .
\end{align}
 Fixing any basis $B_0$ we have that
\begin{align}
\mathcal{B}_k=\{\ \vps:=\rho_{k,\sigma\cdot B_0}\ |\ \sigma\in G:= SL(N_k+1,\mathbb{C}) \} \ .
\end{align}
Therefore there is a natural map 
\begin{align}
G \ra\mathcal{B}_k \ .
\end{align}
In this way we view the Mabuchi energy of $(X , {\om_{FS}}|_X)$ as a function on $G$
\begin{align}
\nu_{\om}(\sigma)=\nu_{\om}(\frac{1}{k}\vps) \ .
\end{align}
 In this paper we have simply let $k=1$, that is, we assume $\mathbb{L}$ is very ample and coincides with $\mathcal{O}(1)|_X$. 
 The importance of these spaces is brought out in the following result, due to G.Tian.
\begin{theorem}(Tian \cite{tian1990})
 \begin{align*}  
 \overline{\bigcup_{k\geq 1}\frac{1}{k}\mathcal{B}_k}=\mathcal{H}_{\om}\quad \mbox{(closure in $C^2$ topology on metrics)}\ .
\end{align*}
\emph{Precisely, given any $\varphi\in\mathcal{H}_{\om}$ there is a sequence of $\frac{1}{k}\rho_{k}(\varphi) \in \frac{1}{k}\mathcal{B}_k$ converging to $\varphi$ in the $C^2$ topology\footnote{It is now known that the convergence takes place in the smooth topology \cite{ruan}, \cite{zelditch} , \cite{catlin}.}.}
\end{theorem}
It is not hard to see that as $k\ra \infty$ we have
\begin{align}
\begin{split}
& \nu_{\om_k}(\rho_{k}(\varphi))\ra \nu_{\om}(\varphi) \\
\ \\
&J_{\om_k}(\rho_{k}(\varphi))\ra  J_{\om}(\varphi) \ .
\end{split}
\end{align}
Where we have defined $\om_k$ by
\begin{align}
\om_k=k\om_{FS}+\frac{\sqrt{-1}}{2\pi}\dl\dlb\log\Big(\sum_j|S_j|_{h^k}^2\Big) \qquad \{S_j\}_{j=0}^{N_k} \ \mbox{a unitary basis}\ .
\end{align}

This justifies Corollary \ref{existence} .
%%%%%%%%%%%%%%%%%%%%%%%%%%%%%%%%%%%%%%%%%%%%%%%%%%%%%%%%%%%%%%%%%%%%%%%%%%%%%%%%%%%%%%%%%%
\section{Appendix II: Hilbert-Mumford vs. Pairs}\label{append2}
We close this paper with a direct comparison of the Hilbert-Mumford (Semi)stability and the (Semi)stability of Pairs. We hope that the table below makes the relationship between the two theories completely transparent. 

 \begin{center} \begin{tabular}{l|l}
  \textbf{Hilbert-Mumford Semistability}& \textbf{Dominance}/\textbf{Semistability of Pairs} \\ \\
\hline \\
 For all $T\leq G$ $\exists\ d\in \mathbb{Z}_{>0}$ and & For all $T\leq G$ and $\chi\in\mathscr{A}_T(v)$\\
 $f\in \mathbb{C}_{\leq d}[\ \mathbb{W}\ ]^T$ such that   & $\exists\ d\in \mathbb{Z}_{>0}$ and $f\in  \mathbb{C}_d[\ \mathbb{V}\oplus\mathbb{W}\ ]^T_{d\chi}$  \\
   $f(w)\neq 0$ and $f(0)=0$ & such that $f((v,w))\neq 0$ and $f|_{ \mathbb{V}}\equiv 0$ \\  \\
 \hline \\
 $0\notin\overline{G\cdot w}$ & \ $\overline{\mathcal{O}}_{vw}\cap\overline{\mathcal{O}}_{v}=\emptyset$  \\ \\
 \hline \\
 $w_{\lambda}(w)\leq 0$   &\ $w_{\lambda}(w)-w_{\lambda}(v)\leq 0$   \\
 for all 1psg's $\lambda$ of $G$ & \ for all 1psg's $\lambda$ of $G$\\ \\
 \hline \\
 $0\in \mathcal{N}(w)$ all $T\leq G$ &\ $\mathcal{N}(v)\subset \mathcal{N}(w)$ {all $T\leq G$} \\ \\
 \hline \\
 $\exists$ $C\geq 0$ such that &\ $\exists$ $C\geq 0$ such that \\
 $\log||\sigma\cdot w||^2\geq -C$   &\ $\log {||\sigma\cdot w||^2}-\log{||\sigma\cdot v||^2}  \geq -C $\\  
  all $\sigma\in G$ &\ all $\sigma\in G$ \\ \\
  \hline\\
  $G\cdot$ \emph{closed} and $G_{w}$ finite & $\exists m\in\mathbb{N}$ such that $(\mathbb{I}^q\otimes v^m,w^{m+1})$ is semistable \\
 \end{tabular} 
  \end{center} 
 \ \\
 \begin{center}\small{Table 1.}\end{center}
 \begin{center}\textbf{Acknowledgements}\end{center}
The author thanks Gang Tian, whose influence is present on virtually every page.
This work was carried out while the author was visiting the Mathematics Department of the Massachusetts Institute of Technology. The author thanks his host, Professor Tomasz Mrowka, for the kind invitation and for many stimulating discussions on the topic of the paper. The key concept of a \emph{strictly stable} pair was inspired by a remark of Dr. Song Sun's. This work was supported in part by a National Science Foundation DMS grant no. 1104448.

 \bibliography{ref}
  \end{document}